\documentclass[12pt]{amsart}

\usepackage{amsmath}
\usepackage{amssymb}
\usepackage{amsthm}
\usepackage{amscd}
\usepackage{xypic}
\swapnumbers \textwidth=16.00cm \textheight=22cm \topmargin=0.00cm
\headsep=1cm \numberwithin{equation}{section}
\def\P{{\mathbb P}}

\newtheorem{theorem}{Theorem}[section]
\newtheorem{lemma}[theorem]{Lemma}

\newtheorem{proposition}[theorem]{Proposition}
\newtheorem{corollary}[theorem]{Corollary}

\theoremstyle{definition}
\newtheorem{definition}[theorem]{Definition}
\newtheorem{remark}[theorem]{Remark}

\newtheorem{convention and reminder}[theorem]{Convention and Reminder}
\newtheorem{convention and remark}[theorem]{Convention and Remark}
\newtheorem{definition and remark}[theorem]{Definition and Remark}

\newtheorem{reminders and definition}[theorem]{Reminders and Definition}

\newtheorem{notation and remarks}[theorem]{Notation and Remarks}
\newtheorem{notation and remark}[theorem]{Notation and Remark}
\newtheorem{example}[theorem]{Example}

\newcommand\Ker{\operatorname{\Ker}}

\def\QR{\mathsf{QR}}

\title{On rank $3$ quadratic equations of projective varieties}

\author[E. Park]{Euisung Park}
\email{euisungpark@korea.ac.kr}
\address{Euisung Park, Department of Mathematics, Korea University, Seoul 136-701, Republic of Korea}

\begin{document}

\keywords{rank of quadratic equation, low rank loci, minimal irreducible decomposition, Veronese embedding}
\subjclass[2010]{14E25, 13C05, 14M15, 14A10}

\begin{abstract}
Let $X \subset \P^r$ be a linearly normal variety defined by a very ample line bundle $L$ on a projective variety $X$. Recently it is shown in \cite{HLMP} that there are many cases where $(X,L)$ satisfies property $\textsf{QR} (3)$ in the sense that the homogeneous ideal $I(X,L)$ of $X$ is generated by quadratic polynomials of rank $3$. The locus $\Phi_3 (X,L)$ of rank $3$ quadratic equations of $X$ in $\P \left( I(X,L)_2 \right)$ is a projective algebraic set, and property $\textsf{QR} (3)$ of $(X,L)$ is equivalent to that $\Phi_3 (X)$ is nondegenerate in $\P \left( I(X)_2 \right)$.

In this paper we study geometric structures of $\Phi_3 (X,L)$ such as its minimal irreducible decomposition. Let
\begin{equation*}
\Sigma (X,L) = \{ (A,B) ~|~  A,B \in {\rm Pic}(X),~L = A^2 \otimes B,~h^0 (X,A) \geq 2,~h^0 (X,B) \geq 1 \}.
\end{equation*}
We first construct a projective subvariety $W(A,B) \subset \Phi_3 (X,L)$ for each $(A,B)$ in $\Sigma (X,L)$. Then we prove that the equality
\begin{equation*}
\Phi_3 (X,L) ~=~ \bigcup_{(A,B) \in \Sigma (X,L)} W(A,B)
\end{equation*}
holds when $X$ is locally factorial. Thus this is an irreducible decomposition of $\Phi_3 (X,L)$ when ${\rm Pic} (X)$ is finitely generated and hence $\Sigma(X,L)$ is a finite set. Also we find a condition that the above irreducible decomposition is minimal. For example, it is a minimal irreducible decomposition of $\Phi_3 (X,L)$ if ${\rm Pic}(X)$ is generated by a very ample line bundle.
\end{abstract}

\maketitle \tableofcontents \setcounter{page}{1}

\section{Introduction}

\noindent We work over an algebraically closed field $\mathbb{K}$ of characteristic $\neq 2$. Let $X$ be a projective variety and $L$ a very ample line bundle on $X$ defining the linearly normal embedding
\begin{equation*}
X \subset \mathbb{P}^r ,\quad r = h^0 (X,L) -1.
\end{equation*}
We denote by $I(X,L)$ the homogeneous ideal of $X$ in the homogeneous coordinate ring $S = \mathbb{K}[x_0 , x_1 , \ldots , x_r ]$ of $\mathbb{P}^r$. It is one of the central issue in algebraic geometry to understand the interaction between geometric properties of $(X,L)$ and structural properties of $I(X,L)$. For the last few decades, the problem of giving conditions to guarantee that $I(X,L)$ is generated by quadrics and the first few syzygy modules are generated by linear syzygies have attracted considerable attention (cf. \cite{Gre1}, \cite{Gre2}, \cite{GL}, \cite{EL}, \cite{GP}, \cite{Ina}, etc). Another important direction to study structures of defining equations of $X$ is to examine whether $I(X,L)$ is defined by $2$-minors of one or several linear determinantal presentations (cf. \cite{EKS}, \cite{Pu}, \cite{Ha}, \cite{B}, \cite{SS}, etc). Also, the finer structure of the finite dimensional $\mathbb{K}$-vector space $I(X,L)_2$ of quadratic defining equations of $X$ in $\mathbb{P}^r$ can be studied through the ranks of its members. Recall that the rank of a nonzero homogeneous quadratic polynomial $Q$ in $S$ is the smallest positive integer $k$ such that $Q$ is the sum of $k$ squares of linear polynomials. Due to \cite{HLMP}, we say that $(X,L)$ satisfies property $\textsf{QR} (k)$ if $I(X,L)$ can be generated by quadratic polynomials of rank at most $k$. There are many examples of property $\textsf{QR} (4)$. If $I(X,L)$ is generated by $2$-minors of some linear determinantal presentations then $(X,L)$ satisfies property $\textsf{QR} (4)$. Many classical constructions in projective geometry, such as rational normal scrolls and Segre-Veronese varieties satisfy property $\textsf{QR} (4)$. We refer the reader to \cite{MP} for property $\textsf{QR} (k)$ of several projective varieties. Let $\mathcal{C}$ be a smooth projective curve of genus $g$ and let $\mathcal{L}$ be a line bundle on $\mathcal{C}$ of degree $d$. In \cite{SD}, B. Saint-Donat proved that if $d \geq 2g+2$ then $(\mathcal{C}, \mathcal{L})$ satisfies property $\QR (4)$. When $\mathcal{C}$ is non-hyperelliptic, non-trigonal and not isomorphic to a plane quintic, M. Green in \cite{Gre} reproved the classical Torelli's Theorem by showing that the canonical embedding of $\mathcal{C}$ satisfies property $\QR (4)$. See also \cite{Pe, SD2, ArH}. Recently it is shown there are many cases where $(X,L)$ satisfies property $\textsf{QR} (3)$.

\begin{theorem}[Theorem 1.1 and Theorem 1.3 in \cite{HLMP}, Theorem 1.1 in \cite{PaCurve}]\label{thm:HLMP}
Suppose that ${\rm char}(\mathbb{K}) \neq 2,3$.

\renewcommand{\descriptionlabel}[1]%
             {\hspace{\labelsep}\textrm{#1}}
\begin{description}
\setlength{\labelwidth}{13mm}
\setlength{\labelsep}{1.5mm}
\setlength{\itemindent}{0mm}

\item[$(1)$] $(\P^n , \mathcal{O}_{\P^n} (d))$ satisfies property $\QR (3)$ for any $n \geq 1$ and $d \geq 2$.

\item[$(2)$] Let $A$ be a very ample line bundle on a projective scheme $X$ defining the linearly normal
embedding
$$X \subset \P H^0 (X,A).$$
If $m$ is an integer such that $X$ is $j$-normal for all $j \geq m$ and $I(X,A)$ is generated by forms of degree $\leq m$, then $(X,A^d )$ satisfies property $\QR (3)$ for all $d \geq m$.

\item[$(3)$] Let $\mathcal{C}$ be a projective integral curve of arithmetic genus $g$ and $\mathcal{L}$ a line bundle on $\mathcal{C}$ of degree $d$. If $d \geq 4g+4$, then $(\mathcal{C},\mathcal{L})$ satisfies property $\QR (3)$.
\end{description}
\end{theorem}

For each $1 \leq k \leq r$, let $\Phi_k (X,L)$ be the locus of all nonzero quadratic polynomials of rank at most $k$ in the projective space $\mathbb{P} (I(X,L)_2 )$. That is,
\begin{equation*}
\Phi_k (X,L) := \{ [Q] ~|~ Q \in I(X,L)_2 - \{0\},~ {\rm rank}(Q) \leq k \} \subset \mathbb{P} (I(X,L)_2 ).
\end{equation*}
Then $\Phi_k (X,L)$ is a projective algebraic set. Indeed, if $\{Q_0 , Q_1 , \ldots , Q_m \}$ is a basis of $I(X,L)_2$ then every member $Q$ of $I(X,L)_2$ is of the form
\begin{equation*}
Q = y_0 Q_0 + y_1 Q_1 + \cdots + y_m Q_m ~ \mbox{for some} ~ y_0 , y_1 , \ldots , y_m \in \mathbb{K}.
\end{equation*}
Now, let $M$ be the $(r+1) \times (r+1)$ symmetric matrix associated to $Q$. Thus its entries are linear forms in $y_0 , y_1 , \ldots , y_m$. If we identify the projective space $\mathbb{P} (I(X,L)_2 )$ with $\P^m$ by sending $[Q]$ to $[y_0 : y_1 : \cdots :y_m ]$, then $\Phi_k (X,L)$ is the common zero set of all $(k+1) \times (k+1)$ minors of $M$ in $\P^m$. In particular, $\Phi_k (X,L)$ is a projective algebraic set. There is an ascending filtration associated to $X \subset \P^r$:
\begin{equation*}
\emptyset = \Phi_2 (X,L) \subset \Phi_3 (X,L) \subset \cdots \subset \Phi_r (X,L) = V ({\mbox det}(M)) \subset  \P^m
\end{equation*}
Here, $\Phi_2 (X,L)$ is empty since $X$ is irreducible and nondegenerate in $\P^r$. Note that $I(X,L)_2$ can be generated by quadratic polynomials of rank at most $k$ if and only if $\Phi_k (X,L)$ is nondegenerate in $\P^m$. \\

The primary goal of this paper is to elucidate the structure of the projective algebraic set $\Phi_3 (X,L)$ such as its minimal irreducible decomposition, i.e. its representation as the union of projective irreducible varieties having no inclusion relation to each other, the dimensions and the degrees of its irreducible components, etc.

\begin{theorem}\label{thm:main1}
Let $(X,L)$ be as above and let
\begin{equation*}
\Sigma (X,L) = \{ (A,B) ~|~  A,B \in {\rm Pic}(X),~L = A^2 \otimes B,~h^0 (X,A) \geq 2,~h^0 (X,B) \geq 1 \}.
\end{equation*}
Then
\smallskip

\renewcommand{\descriptionlabel}[1]%
             {\hspace{\labelsep}\textrm{#1}}
\begin{description}
\setlength{\labelwidth}{13mm}
\setlength{\labelsep}{1.5mm}
\setlength{\itemindent}{0mm}

\item[{\rm (1)}] For each $(A,B) \in \Sigma (X,L)$, there is a projective subvariety $W(A,B) \subset \Phi_3 (X,L)$ which admits a finite surjective morphism
\begin{equation*}
\widetilde{Q_{A,B}} : \mathbb{G} (1,\P^p ) \times \P^q \rightarrow W(A,B)
\end{equation*}
where $p = h^0 (X,A)-1$ and $q= h^0 (X,B)-1$. Thus $\dim~W(A,B) = 2p + q -2$.
\item[{\rm (2)}] Suppose that $X$ is locally factorial. Then $\Phi_3 (X,L)$ is decomposed as
\begin{equation}\label{eq:irr decomp}
\Phi_3 (X,L) ~=~ \bigcup_{(A,B) \in \Sigma (X,L)} W(A,B)
\end{equation}
Moreover, this is an irreducible decomposition of $\Phi_3 (X,L)$ if ${\rm Pic}(X)$ is finitely generated.
\end{description}
\end{theorem}

See section 2 for the precise construction of $W(A,B)$ and for the proof of Theorem \ref{thm:main1}.(1). Also see Proposition \ref{prop:decomp by irreducibles} for the proof of Theorem \ref{thm:main1}.(2).

\begin{remark}
In Theorem \ref{thm:main1}.(2), there are the conditions that $X$ is locally factorial and that ${\rm Pic}(X)$ is finitely generated. The former condition means that all Weil divisors of $X$ are Cartier. In Proposition \ref{prop:decomp by irreducibles} we prove that if $X$ is locally factorial then every rank $3$ quadratic polynomial in $I(X,L)$ is contained in $W(A,B)$ for some $(A,B) \in \Sigma (X,L)$. Also the latter condition implies that $\Sigma (X,L)$ is a finite set and hence the right hand side of (\ref{eq:irr decomp}) is a finite union.
\end{remark}

Perhaps the most powerful application of Theorem \ref{thm:main1} is to the case where ${\rm Pic}(X)$ is generated by a very ample line bundle, say $A$. In such a case, $\Sigma (X,A)$ is empty and hence Theorem \ref{thm:main1} says that $X \subset \P H^0 (X,A)$ satisfies no quadratic polynomials of rank $3$. Also since
$$\Sigma (X,A^2) = \{ (A, \mathcal{O}_X ) \} \quad \mbox{and} \quad \Sigma (X,A^3) = \{ (A, A) \},$$
it follows by Theorem \ref{thm:main1} that $\Phi_3 (X,A^2 ) = W(A,\mathcal{O}_X )$ and $\Phi_3 (X,A^3 ) = W(A,A)$ are projective varieties of dimensions $2h^0 (X,A)-4$ and $3h^0 (X,A)-5$, respectively.

When $(X,L) = (\P^n ,\mathcal{O}_{\P^n} (2))$ we obtain the following much simpler structure of $\Phi_3$.

\begin{corollary}\label{cor:second Veronese embedding}
Suppose that $\mbox{char}(\mathbb{K}) \neq 2,3$. Then $\Phi_3 (\P^n ,\mathcal{O}_{\P^n} (2))$ is projectively equivalent to the second Veronese embedding of the Pl\"{u}cker embedding $\mathbb{G} (1,\P^n ) \subset \P^{{{n+1} \choose {2}} -1} $.
\end{corollary}

This result implies that $\Phi_3 (\P^2 ,\mathcal{O}_{\P^2} (2))$ is projectively equivalent to the Veronese surface $\nu_2 (\P^2 ) \subset \P^5$. In \cite[Theorem 1.1]{PaSMD}, there is a computational proof of this result by using Macaulay2. Thus Corollary \ref{cor:second Veronese embedding} enables us to understand this fact theoretically.\\

The next natural question is when is the irreducible decomposition of $\Phi_3 (X,L)$ in Theorem \ref{thm:main1}.(2) minimal? To answer for this question, we begin with the following definition.

\begin{definition}\label{def:irreducible}
Suppose that $X$ is locally factorial and $\mbox{Pic}(X)$ is finitely generated. Let
\begin{equation*}
\mbox{Cox}(X) = \bigoplus_{\mathcal{L} \in {\rm Pic}(X)} H^0 (X,\mathcal{L})
\end{equation*}
be the Cox ring of $X$. Let $(X,L)$ be as in Theorem \ref{thm:main1}. We say that an element $(A,B) \in \Sigma (X,L)$ is \textit{irreducible} if a general member of $H^0 (X,A)$ is irreducible in $\mbox{Cox}(X)$ and if either $B = \mathcal{O}_X$ or else a general member of $H^0 (X,B)$ is irreducible in $\mbox{Cox}(X)$.
\end{definition}

Recall that $\mbox{Cox}(X)$ is a unique factorization domain (cf. \cite{Ar}, \cite{BH} or \cite{EKW}).

\begin{theorem}\label{thm:main2}
Let $(X,L)$ be as in Definition \ref{def:irreducible}. If $(A,B) \in \Sigma (X,L)$ is irreducible then $W(A,B)$ is an irreducible component of $\Phi_3 (X,L)$. Therefore the irreducible decomposition of $\Phi_3 (X,L)$ in Theorem \ref{thm:main1}.(2) is minimal if every $(A,B) \in \Sigma (X,L)$ is irreducible.
\end{theorem}
\smallskip

Again the most powerful application of Theorem \ref{thm:main2} is to the case where ${\rm Pic}(X)$ is generated by a very ample line bundle, say $A$.

\begin{corollary}\label{cor:application Pic Z case}
Suppose that $X$ is a locally factorial variety of dimension $\geq 2$ and ${\rm Pic}(X)$ is generated by a very ample line bundle, say $A$. Then for every $d \geq 2$, the minimal irreducible decomposition of $\Phi_3 (X,A^d)$ is
\begin{equation*}
\Phi_3 (X,A^d ) = \bigcup_{1 \leq \ell \leq d/2} W(A^{\ell},A^{d-2\ell} ).
\end{equation*}
\end{corollary}
\smallskip

When $(X,L) = (\P^n , \mathcal{O}_{\P^n } (d))$ this result and Theorem \ref{thm:main1} show that the minimal irreducible decomposition of $\Phi_3 (X,L)$ is of the form
\begin{equation*}
\Phi_3 (\P^n , \mathcal{O}_{\P^n } (d) ) = \bigcup_{1 \leq \ell \leq d/2} W(\mathcal{O}_{\P^n} (\ell) ,\mathcal{O}_{\P^n} (d-2\ell) )
\end{equation*}
and $W(\mathcal{O}_{\P^n} (\ell) ,\mathcal{O}_{\P^n} (d-2\ell) )$ is a projective variety of dimension $2 {{n+\ell} \choose {n}} + {{n+d-2\ell} \choose {n}} -5$.

If $X$ is a locally factorial curve such that ${\rm Pic}(X)$ is simply generated, then $X$ must be the projective line. In this case, Theorem \ref{thm:main1} shows that
\begin{equation*}
\Phi_3 (\P^1 , \mathcal{O}_{\P^1 } (d) ) = \bigcup_{1 \leq \ell \leq d/2} W(\mathcal{O}_{\P^1} (\ell) ,\mathcal{O}_{\P^1} (d-2\ell) )
\end{equation*}
is an irreducible decomposition of $\Phi_3 (\P^1 , \mathcal{O}_{\P^1 } (d) )$ and $W(\mathcal{O}_{\P^1} (\ell) ,\mathcal{O}_{\P^1} (d-2\ell) )$ is a $(d-2)$-dimensional variety for every $1 \leq \ell \leq d/2$. In this case, the minimality of this irreducible decomposition is shown to be true in \cite{PaShim}. \\

The organization of this paper is as follows. In Section 2 we construct the set $W(A,B)$ mentioned in Theorem \ref{thm:main1} and prove that it is a projective variety. Section 3 is devoted to describing the minimal irreducible decomposition of $\Phi_3 (X,L)$. Finally, in section $4$ we provide an example where $X$ is an irrational curve and hence ${\rm Pic}(X)$ is not finitely generated. \\

\noindent {\bf Acknowledgement.} This work was supported by the National Research Foundation of Korea(NRF) grant funded by the Korea government(MSIT) (No. 2022R1A2C1002784). The author is also grateful to the referee for the valuable comments and helpful corrections.
\vspace{0.2 cm}

\section{Construction of $W(A,B)$}
\noindent In this section we will construct the set $W(A,B)$ mentioned in Theorem \ref{thm:main1}. Let $X$ be a projective variety and $L$ a very ample line bundle on $X$ defining the linearly normal embedding
\begin{equation*}
X \subset \P^r ,\quad r = h^0 (X,L) -1.
\end{equation*}
Throughout this section we suppose that $L$ is decomposed into $L = A^2 \otimes B$ for some $A,B \in \mbox{Pic}(X)$ such that $p := h^0 (X,A)-1 \geq 1$ and $q := h^0 (X,B)-1 \geq 0$. We use the following notations:\\

$\bullet$ $\varphi : H^0 (X,L) \rightarrow S_1$ : the isomorphism inducing the embedding of $X$ by $L$
\smallskip

$\bullet$ $\{Q_0 , Q_1 , \ldots , Q_m \}$ : a basis for $I(X,L)_2$
\smallskip

$\bullet$ $\{s_0 , s_1 ,\ldots , s_p\}$ : a basis for $H^0 (X,A)$
\smallskip

$\bullet$ $\{h_0 , h_1 , \ldots , h_q \}$ : a basis for $H^0 (X,B)$
\smallskip

$\bullet$ $V(A,B) = H^0 (X,A) \times H^0 (X,A) \times H^0 (X,B)$ \\

\begin{definition and remark}\label{def:W(A,B)}
We define the map $Q_{A,B} : V(A,B) \rightarrow I(X,L)_2$ by
\begin{equation*}
Q_{A,B} (s,t,h) = \varphi (s \otimes s \otimes h) \varphi (t \otimes t \otimes h) - \varphi (s \otimes t \otimes h)^2 .
\end{equation*}
Also we denote the image of $Q_{A,B}$ in $\P (I(X)_2 )$ by $W(A,B)$. That is,
\begin{equation*}
W(A,B) := \{ [Q_{A,B} (s,t,h)]~|~(s,t,h) \in V(A,B) ,~Q_{A,B} (s,t,h) \neq 0 \}.
\end{equation*}
The map $Q_{A,B}$ is well-defined since the restriction of the quadratic polynomial $Q_{A,B} (s,t,h)$ to $X$ becomes
\begin{equation*}
(s \otimes s \otimes h)|_X \times (t \otimes t \otimes h)|_X - (s \otimes t \otimes h)^2 |_X = 0
\end{equation*}
(cf. \cite[Proposition 6.10]{E}). Thus $W (A,B)$ is contained in $\Phi_3 (X)$.
\end{definition and remark}
\smallskip

We will show that $W(A,B)$ is a projective variety of dimension $2p+q-2$. We begin with the following lemma.

\begin{lemma}\label{lem:basic properties}
Let $(s,t,h) \in V(A,B)$ and $a,b,c,d \in \mathbb{K}$. Then
\smallskip

\renewcommand{\descriptionlabel}[1]%
             {\hspace{\labelsep}\textrm{#1}}
\begin{description}
\setlength{\labelwidth}{13mm}
\setlength{\labelsep}{1.5mm}
\setlength{\itemindent}{0mm}

\item[$(1)$] $Q_{A,B} (s,t,h) = Q_{A,B} (t,s,h)$ and $Q_{A,B} (as,bt,ch) = a^2 b^2 c^2 Q_{A,B} (s,t,h)$
\smallskip

\item[$(2)$] $Q_{A,B} (s,t,h)=0$ if and only if either $\{ s,t \}$ is $\mathbb{K}$-linearly dependent or else $h=0$.
\smallskip

\item[$(3)$] $Q_{A,B} (as+bt,cs+dt,h) = (ad-bc)^2 Q_{A,B} (s,t,h)$.
\end{description}
\end{lemma}

\begin{proof}
$(1)$ The assertions come immediately from the definition of the map $Q_{A,B}$.

\noindent $(2)$ $(\Rightarrow)$ If $s$ or $t$ or $h$ is $0$, then we are done. Suppose that $s$, $t$ and $h$ are nonzero. Then $\varphi (s^2 h)$, $\varphi (t^2 h)$ and $\varphi (sth)$ are nonzero linear forms on $\mathbb{P}^r$. If $Q_{A,B} (s,t,h)=0$ then
\begin{equation*}
\varphi (s^2 h) \varphi (t^2 h) = \varphi (sth)^2
\end{equation*}
and hence
$$\varphi (sth) = \lambda \varphi (s^2 h) = \lambda^{-1} \varphi (t^2 h)$$
for some $\lambda \in \mathbb{K}^*$. Thus
\begin{equation*}
sth = \lambda s^2 h = \lambda^{-1} t^2 h
\end{equation*}
and hence $t = \lambda s$. This shows that $\{ s,t \}$ is $\mathbb{K}$-linearly dependent.

\noindent $(\Leftarrow)$ If $s$ or $t$ or $h$ is $0$, then $Q_{A,B} (s,t,h) =0$. Now, suppose that $s$, $t$ and $h$ are nonzero and $\{s,t\}$ is $\mathbb{K}$-linearly dependent. Then $t = \lambda s$ for some $\lambda \in \mathbb{K}^*$ and hence
\begin{equation*}
Q_{A,B} (s,t,h) = \varphi (s^2 h) \varphi (\lambda^2 s^2 h) - \varphi (\lambda s^2 h)^2 = 0.
\end{equation*}

\noindent $(3)$ When $c=0$ one can easily check that
\begin{equation}\label{eq:det form}
Q_{A,B} (as+bt,dt,h) = a^2 d^2 Q_{A,B} (s,t,h).
\end{equation}
Suppose that $c \neq 0$ and let $u = cs+dt$. Then it holds by (\ref{eq:det form}) that
\begin{equation*}
Q_{A,B} (as+bt,cs+dt,h) = Q_{A,B} \left( \frac{a}{c} u + \frac{bc-ad}{c} t , u , h \right) = \frac{(bc-ad)^2}{c^2 } Q_{A,B} (t,u,h).
\end{equation*}
Also, again by (\ref{eq:det form}), we get
\begin{equation*}
Q_{A,B} ( t , u , h) = Q_{A,B} (cs+dt,t,h) = c^2 Q_{A,B} (s,t,h).
\end{equation*}
Thus the desired equality $Q_{A,B} (as+bt,cs+dt,h) = (ad-bc)^2 Q_{A,B} (s,t,h)$ holds.
\end{proof}

By Lemma \ref{lem:basic properties}.$(2)$, it holds that the inverse image of $I(X,L)_2 - \{0\}$ is equal to
\begin{equation*}
Q_{A,B} ^{-1} \left( I(X,L)_2 - \{0\} \right) = \{ (s,t)~|~s,t \in H^0 (X,A),~\dim_{\mathbb{K}} \langle s,t \rangle = 2 \} \times \left( H^0 (X,B) - \{0\} \right).
\end{equation*}
So, there are the following natural maps
$$\pi_1 : Q_{A,B} ^{-1} \left( I(X,L)_2 - \{0\} \right) \rightarrow \mathbb{G}(1,\P^p ) \times \P^q \quad \mbox{and} \quad \pi_2 : I(X,L)_2 - \{0\} \rightarrow \P \left( I(X,L)_2 \right) = \P^m .$$
Also Lemma \ref{lem:basic properties}.$(1)$ and $(3)$ show that $Q_{A,B}$ induces a map
$$\widetilde{Q_{A,B}} : \mathbb{G}(1,\P^p ) \times \P^q \rightarrow \P^m$$
such that the following diagram is commutative:\\
\begin{equation*}
\begin{CD}
Q_{A,B} ^{-1} \left( I(X,L)_2 - \{0\} \right) & \quad \stackrel{Q_{A,B}}{\rightarrow} \quad  &  I(X,L)_2 - \{0\}         \\
\pi_1 \downarrow                        &                                          & \quad  \downarrow  \pi_2      \\
\mathbb{G}(1,\P^p ) \times \P^q  & \quad  \stackrel{\widetilde{Q_{A,B}}}{\rightarrow}  \quad  &  \P^m \\\\
\end{CD}
\end{equation*}
Recall that $W(A,B)$ denotes the image of $Q_{A,B}$ in $\P^m$. Thus we can interpret $W(A,B)$ as the image of the projective variety $\mathbb{G}(1,\P^p ) \times \P^q$ by $\widetilde{Q_{A,B}}$. Now, it is important to know whether $\widetilde{Q_{A,B}}$ is a morphism or not. The following theorem answers this question.

\begin{theorem}\label{thm:main1 rank3 case}
The map $\widetilde{Q_{A,B}} : \mathbb{G} (1,\P^p ) \times \P^q \rightarrow \P^m$ is a finite morphism such that
\begin{equation*}
\left( \widetilde{Q_{A,B}} \right)^* (\mathcal{O}_{\P^m } (1) ) = \mathcal{O}_{\mathbb{G} (1,\P^p ) \times \P^q } (2,2).
\end{equation*}
In particular, $W(A,B)$ is a projective variety and $\dim~W(A,B) = 2p+q-2$.
\end{theorem}

To prove this theorem we require some notations and definitions.

\begin{notation and remarks}\label{no and rks:Grassmannian}
$(1)$ An element $(s,t,h)$ of $V(A,B)$ is uniquely written as
\begin{equation}\label{eq:(s,t,h)}
(s,t,h) = (x_0 s_0 +  \cdots +x_p s_p  , ~ y_0 s_0 +  \cdots + y_p s_p , z_0 h_0  + \cdots +z_q h_q ).
\end{equation}
The space $V(A,B)$ can be identified with $M_{2 \times (p+1)} (\mathbb{K}) \times \mathbb{A}^{q+1}$
via the map sending
\begin{equation*}
(x_0 s_0 +  \cdots +x_p s_p  , ~ y_0 s_0 +  \cdots + y_p s_p , z_0 h_0  + \cdots +z_q h_q ) \in V(A,B)
\end{equation*}
to
$$\left( \left(\begin{array}{rrrr}
x_0 & x_1 & \cdots & x_p \\
y_0 & y_1 & \cdots & y_p
\end{array}\right) , (z_0 , z_1 , \ldots, z_q ) \right) \in M_{2 \times (p+1)} (\mathbb{K}) \times \mathbb{A}^{q+1} .$$
Let the affine coordinate ring of the affine space $V(A,B) \cong M_{2 \times (p+1)} (\mathbb{K}) \times \mathbb{A}^{q+1}$ be
$$T = \mathbb{K} [x_0 ,x_1 , \ldots , x_p , y_0 ,y_1 , \ldots , y_p ] \otimes_{\mathbb{K}} \mathbb{K}[z_0 , z_1 , \ldots , z_q ].$$
Consider the action of ${\rm SL}_2 (\mathbb{K})$ on $M_{2 \times (p+1)} (\mathbb{K})$ via multiplication on the left. Then, by \textit{the first fundamental theorem of invariant theory}, it holds that
\begin{equation*}
T^{{\rm SL}_2 (\mathbb{K})} = \mathbb{K} [ \{x_i y_j - x_j y_i ~|~ 0 \leq i < j \leq p \} ] \otimes_{\mathbb{K}} \mathbb{K} [z_0 ,z_1 , \ldots , z_q]
\end{equation*}
(cf. \cite[Theorem 8.6]{Muk}). The ring
$$\mathbb{K} [ \{x_i y_j - x_j y_i ~|~ 0 \leq i < j \leq p \} ]$$
is the homogeneous coordinate ring of the Pl\"{u}cker embedding of the Grassmannian manifold $\mathbb{G} (1,\P^p )$ into $\P^{{{p+1} \choose 2} -1}$. Also $\mathbb{K} [z_0 ,z_1 , \ldots , z_q]$ is the homogeneous coordinate ring of the projective space $\P^q$. For each line bundle $\mathcal{O} (a,b)$ on $\mathbb{G} (1,\P^p ) \times \P^q$, it holds that
\begin{equation*}
H^0 (\mathbb{G} (1,\P^p ) \times \P^q , \mathcal{O} (a,b)) \cong  \mathbb{K} [ \{x_i y_j - x_j y_i ~|~ 0 \leq i < j \leq p \} ]_a \otimes_{\mathbb{K}} \mathbb{K} [z_0 ,z_1 , \ldots , z_q]_b .
\end{equation*}

\noindent $(2)$ Let $(s,t,h) \in V(A,B)$ be as in (\ref{eq:(s,t,h)}). Since $Q_{A,B} (s,t,h)$ is a member of $I(X,L)_2$, it can be expressed uniquely as
\begin{equation*}
Q_{A,B} (s,t,h) = G_0 Q_0 + G_1 Q_1 + \cdots + G_m Q_m
\end{equation*}
for some functions $G_0 , G_1 , \ldots , G_m$ in $x_0 ,x_1 , \ldots , x_p , y_0 ,y_1 , \ldots , y_p , z_0 , z_1 , \ldots , z_q$. Note that the map
$$\widetilde{Q_{A,B}} : \mathbb{G} (1,\P^p ) \times \P^q \rightarrow \P^m$$
is defined as
\begin{equation*}
\widetilde{Q_{A,B}} (\langle s,t \rangle , [h]) = [G_0 (P) : G_1 (P) : \cdots : G_m (P)]
\end{equation*}
where $P = (x_0 ,x_1 , \ldots , x_p , y_0 ,y_1 , \ldots , y_p , z_0 , z_1 , \ldots , z_q)$. Therefore it is important to understand the functions $G_0 , G_1 , \ldots , G_m$ in order to prove Theorem \ref{thm:main1 rank3 case}.

\noindent $(3)$ (cf. \cite[Lemma 2.3.(2)]{HLMP}) For any $s,t,u\in H^0(X,A)$, $g,h \in H^0(X,B)$ and $a,b \in \mathbb{K}$, it holds that
\begin{equation*}
Q_{A,B} (s,at+bu,h) = (a^2 -ab) Q_{A,B} (s,t,h) + (b^2 -ab) Q_{A,B} (s,u,h) + ab Q_{A,B} (s,t+u,h)
\end{equation*}
and
\begin{equation*}
Q_{A,B} (s,t,ag+bh) = (a^2 -ab) Q_{A,B} (s,t,g) + (b^2 -ab) Q_{A,B} (s,t,h) + ab Q_{A,B} (s,t,g+h).
\end{equation*}

\noindent $(4)$ (cf. \cite[Lemma 2.3.(3)]{HLMP}) For $\ell \geq 3$, let
$$s,t,t_1 , \ldots , t_{\ell} \in H^0(X,A) \quad \mbox{and} \quad h,g_1 , \ldots , g_{\ell} \in H^0(X,B).$$
Then
\begin{equation}\label{eq:expanding 1}
Q_{A,B} (s,t_1+t_2+\cdots+t_\ell ,h)=\sum_{1\leq i<j\leq \ell} Q_{A,B} (s,t_i+t_j,h)
-(\ell -2)\sum_{i=1}^{\ell} Q_{A,B} (s,t_i,h)
\end{equation}
and
\begin{equation}\label{eq:expanding 2}
Q_{A,B}(s,t,g_1+g_2+\cdots+g_{\ell})=\sum_{1\leq i<j\leq \ell } Q_{A,B} (s,t,g_i+g_j)
-(\ell -2)\sum_{i=1}^{\ell} Q_{A,B} (s,t,g_i).
\end{equation}
\end{notation and remarks}
\smallskip

\begin{proposition}\label{prop:coeff poly}
Let $G_0 , G_1 , \ldots , G_m$ be as in Notation and Remarks \ref{no and rks:Grassmannian}. Then
$$\{ G_0 , G_1 , \ldots , G_m \} \subset \mathbb{K} [x_0 , x_1 , \ldots ,x_p ]_2 \otimes_{\mathbb{K}} \mathbb{K} [y_0 , y_1 , \ldots ,y_p ]_2 \otimes_{\mathbb{K}} \mathbb{K} [z_0 , \ldots , z_q]_2 .$$
\end{proposition}

\begin{proof}
Let $(s,t,h) \in V(A,B)$ be as in (\ref{eq:(s,t,h)}). Also let
\begin{equation*}
\begin{CD}
I &~=~& \{ x_i s_i ~|~ 0 \leq i \leq p \} \cup \{ x_i s_i + x_j s_j ~|~ 0 \leq i < j \leq p \},\\
I'&~=~& \{   s_i ~|~ 0 \leq i \leq p \} \cup \{   s_i +  s_j ~|~ 0 \leq i < j \leq p \}, \\
J &~=~& \{ y_i s_i ~|~ 0 \leq i \leq p \} \cup \{ y_i s_i + y_j s_j ~|~ 0 \leq i < j \leq p \}, \\
J'&~=~& \{   s_i ~|~ 0 \leq i \leq p \} \cup \{   s_i +   s_j ~|~ 0 \leq i < j \leq p \}, \\
K &~=~& \{ z_i h_i ~|~  0 \leq i \leq q \} \cup \{ z_i h_i + z_j h_j ~|~ 0 \leq i < j \leq q \} ~\mbox{and}\\
K'&~=~& \{  h_i ~|~  0 \leq i \leq q \} \cup \{  h_i + h_j ~|~ 0 \leq i < j \leq q \}.
\end{CD}
\end{equation*}
By expanding $Q_{A,B} (s,t,h)$ using (\ref{eq:expanding 1}) and (\ref{eq:expanding 2}), we can write $Q_{A,B} (s,t,h)$ as a $\mathbb{K}$-linear combination of elements in the set
\begin{equation*}
\Sigma = \{ Q_{A,B} (s,t,h) ~|~s \in I ,~t \in J , ~h \in K \}.
\end{equation*}
Then, by Notation and Remarks \ref{no and rks:Grassmannian}.$(3)$, every element in $\Sigma$ can be written as a linear combination of elements in the set
\begin{equation*}
\Sigma' = \{ Q_{A,B} (s,t,h) ~|~s \in I' ,~t \in J' , ~h \in K' \}.
\end{equation*}
with coefficients in $\mathbb{K} [x_0 , x_1 , \ldots ,x_p ]_2 \otimes_{\mathbb{K}} \mathbb{K} [y_0 , y_1 , \ldots ,y_p ]_2 \otimes_{\mathbb{K}} \mathbb{K} [z_0 , \ldots , z_q]_2$. For example,
\begin{equation*}
\begin{CD}
Q_{A,B} (x_0 s_0 + x_1 s_1 , y_2 s_2 , z_0 h_0 +z_1 h_1  ) &\quad  = \quad& (x_0 ^2 -x_0 x_1 )y_2 ^2 Q_{A,B} (s_0 ,s_2 ,  z_0 h_0 +z_1 h_1) \\
& & + (x_1 ^2 - x_0 x_1 ) y_2 ^2 Q_{A,B} (s_1 ,s_2 ,  z_0 h_0 +z_1 h_1) \\
& & + x_0 x_1 y_2 ^2 Q_{A,B} (s_0 + s_1 , s_2 , z_0 h_0 +z_1 h_1  ).
\end{CD}
\end{equation*}
Also the quadratic polynomials
$$Q_{A,B} (s_0 ,s_2 ,  z_0 h_0 +z_1 h_1),~Q_{A,B} (s_1 ,s_2 ,  z_0 h_0 +z_1 h_1) \quad \mbox{and} \quad Q_{A,B} (s_0 + s_1 , s_2 , z_0 h_0 +z_1 h_1  )$$
are expanded into the linear combinations of elements in $\Sigma'$ with coefficients in $\mathbb{K} [z_0 , \ldots , z_q]_2$. For example,
\begin{equation*}
\begin{CD}
Q_{A,B} (s_0 ,s_2 ,  z_0 h_0 +z_1 h_1) & \quad = \quad & (z_0 ^2 - z_0 z_1 ) Q_{A,B} (s_0 ,s_2 , h_0 ) + (z_1 ^2 - z_0 z_1 ) Q_{A,B} (s_0 , s_2 , h_1 ) \\
& & \quad \quad \quad \quad \quad \quad \quad \quad \quad \quad \quad \quad + z_0 z_1 Q_{A,B} (s_0 ,s_2 , h_0 + h_1 ).
\end{CD}
\end{equation*}
Finally, each element in $\Sigma'$ can be written as a $\mathbb{K}$-linear combination of $Q_0 , Q_1 , \ldots , Q_m$. In consequence, if we write
$$Q_{A,B} (s,t,h) = G_0 Q_0 + G_1 Q_1 + \cdots + G_m Q_m$$
then $G_0 , G_1 , \ldots , G_m$ must be contained in
$$\mathbb{K} [x_0 , x_1 , \ldots ,x_p ]_2 \otimes_{\mathbb{K}} \mathbb{K} [y_0 , y_1 , \ldots ,y_p ]_2 \otimes_{\mathbb{K}} \mathbb{K} [z_0 , \ldots , z_q]_2 .$$
This completes the proof.
\end{proof}

Now we are ready to give a proof of Theorem \ref{thm:main1 rank3 case}.\\

\noindent {\bf Proof of Theorem \ref{thm:main1 rank3 case}.} By Proposition \ref{prop:coeff poly}, it holds that $G_0 , G_1 , \ldots , G_m$ are contained in
$$\mathbb{K} [x_0 , x_1 , \ldots ,x_p ]_2 \otimes_{\mathbb{K}} \mathbb{K} [y_0 , y_1 , \ldots ,y_p ]_2 \otimes_{\mathbb{K}} \mathbb{K} [z_0 , \ldots , z_q]_2 .$$
Also Lemma \ref{lem:basic properties}.$(3)$ says that
\begin{equation*}
Q_{A,B} (as+bt,cs+dt,h) = Q_{A,B} (s,t,h)
\end{equation*}
for any $(s,t,h) \in V(A,B)$ and $\left(\begin{array}{rr}
a & b \\
c & d
\end{array}\right) \in {\rm SL}_2 (\mathbb{K})$. Therefore $G_0 , G_1 , \ldots , G_m$ are contained in the intersection
\begin{equation*}
T^{{\rm SL}_2 (\mathbb{K})} \cap \mathbb{K} [x_0 , x_1 , \ldots ,x_p ]_2 \otimes_{\mathbb{K}} \mathbb{K} [y_0 , y_1 , \ldots ,y_p ]_2 \otimes_{\mathbb{K}} \mathbb{K} [z_0 , \ldots , z_q]_2
\end{equation*}
which is equal to
\begin{equation*}
 \mathbb{K} [ \{x_i y_j - x_j y_i ~|~ 0 \leq i < j \leq p \} ]_2 \otimes_{\mathbb{K}} \mathbb{K} [z_0 ,z_1 , \ldots , z_q]_2  \cong H^0 (\mathbb{G} (1,\P^p ) \times \P^q , \mathcal{O} (2,2))
\end{equation*}
by Notation and Remarks \ref{no and rks:Grassmannian}.$(1)$. Moreover,
$$\{G_0 , G_1 , \ldots , G_m \}$$
is base point free by Lemma \ref{lem:basic properties}.$(2)$. Since the line bundle $\mathcal{O} (2,2)$ of $\mathbb{G} (1,\P^p ) \times \P^q$ is very ample, this shows that the map
$$\widetilde{Q_{A,B}} : \mathbb{G} (1,\P^p ) \times \P^q \rightarrow \P^m$$
is a finite morphism and
\begin{equation*}
\widetilde{Q_{A,B}}^* (\mathcal{O}_{\P^m } (1) ) = \mathcal{O}_{\mathbb{G} (1,\P^p ) \times \P^q } (2,2).
\end{equation*}
It follows that $W(A,B)$, the image of $\widetilde{Q_{A,B}}$ in $\P^m$, is a projective variety and
$$\dim~W(A,B) = \dim~\mathbb{G} (1,\P^p ) \times \P^q = 2p+q-2.$$
This completes the proof.                      \qed

\section{Irreducible decomposition of $\Phi_3 (X)$}
\noindent Throughout this section we suppose that the projective variety $X$ is locally factorial and hence all Weil divisors of $X$ are Cartier. Let $L$ be a very line bundle on $X$ defining the linearly normal embedding
\begin{equation*}
X \subset \mathbb{P}^r ,\quad r = h^0 (X,L) -1.
\end{equation*}
Recall that $\Sigma (X,L)$ is defined to be $\{ (A,B) ~|~ (A,B) \in \Sigma (X),~L = A^2 \otimes B \}$. \\

\begin{proposition}\label{prop:decomp by irreducibles}
Let $(X,L)$ be as above. Then
\begin{equation}\label{eq:natural irred decomp}
\Phi_3 (X,L) = \bigcup_{(A,B) \in \Sigma (X,L)} W(A,B).
\end{equation}
\end{proposition}

\begin{proof}
$(1)$ We know already that $W(A,B) \subset \Phi_3 (X,L)$ for every $(A,B) \in \Sigma (X,L)$. Thus it holds that
\begin{equation*}
\bigcup_{(A,B) \in \Sigma (X,L)} W(A,B) ~\subseteq ~ \Phi_3 (X,L).
\end{equation*}
Now let $Q \in I(X,L)_2$ be an element of rank $3$. Then $Q = xy-z^2$ for some linear forms $x,y,z$ on $\P^r$. Write the Weil divisors $\mbox{div}(z)$, $\mbox{div}(x)$ and $\mbox{div}(y)$ as
\begin{equation*}
\mbox{div}(z) = \sum_{i=1} ^m  d_i D_i ,\quad  \mbox{div}(x) = \sum_{i=1} ^m e_i D_i  \quad \mbox{and} \quad  \mbox{div}(y) = \sum_{i=1} ^m f_i D_i
\end{equation*}
where $D_1 , \ldots , D_m$ are prime divisors on $X$. Suppose that $e_1 , \ldots , e_{\ell}$ are odd and $e_{\ell+1} , \ldots , e_m$ are even. Since $xy = z^2$, it follows that $f_i$ is odd if and only if $1\leq i \leq \ell$. Therefore
\begin{equation*}
L \otimes \mathcal{O}_X (-D_1 - \cdots - D_{\ell} ) = A^2
\end{equation*}
where $H^0 (X,A)$ contains two sections, say $s$ and $t$, such that
\begin{equation*}
\mbox{div}(s) = \sum_{i=1} ^{\ell}  \frac{e_i -1}{2} D_i + \sum_{i=\ell+1} ^m  \frac{e_i}{2} D_i \quad \mbox{and} \quad \mbox{div}(t) =  \sum_{i=1} ^{\ell}  \frac{f_i -1}{2} D_i + \sum_{i=\ell+1} ^m  \frac{f_i}{2} D_i .
\end{equation*}
Now, let $B$ denote the line bundle $\mathcal{O}_X ( D_1 + \cdots + D_{\ell} )$ and let $h \in H^0 (X,B)$ be the element corresponding to the divisor $D_1 + \cdots + D_{\ell}$. Then we have
\begin{equation*}
L = A^2 \otimes B \quad \mbox{and} \quad Q = Q_{A,B} (s,t,h).
\end{equation*}
Since $Q$ is nonzero it follows that $s$ and $t$ are $\mathbb{K}$-linearly independent and $h$ is nonzero. Thus $h^0 (X,A) \geq 2$ and $h^0 (X,B) \geq 1$. Therefore it is shown that $[Q] \in W(A,B)$. This completes the proof of the desired equality in (\ref{eq:natural irred decomp}).
\end{proof}
\smallskip

From now on we suppose that ${\rm Pic}(X)$ is a finitely generated free group. Then the Cox ring of $X$ defined as
\begin{equation*}
{\rm Cox}(X) = \bigoplus_{\mathcal{L} \in {\rm Pic}(X)} H^0 (X,\mathcal{L})
\end{equation*}
is a unique factorization domain (cf. \cite{Ar}, \cite{BH} or \cite{EKW}).

\begin{proposition}\label{prop:distinguishing}
Let $(X,L)$ be as above and let $(A,B) \in \Sigma (X,L)$ be such that there exist $s,t \in H^0 (X,A)$ and $h \in H^0 (X,B)$ where
\begin{enumerate}
\item[$(i)$] $s$ and $t$ are irreducible elements in ${\rm Cox}(X)$,
\item[$(ii)$] either $B = \mathcal{O}_X$ or else $h$ is an irreducible element in ${\rm Cox}(X)$ and
\item[$(iii)$] $Q_{A,B} (s,t,h) \neq 0$.
\end{enumerate}
Then the following statements hold:
\smallskip

\renewcommand{\descriptionlabel}[1]%
             {\hspace{\labelsep}\textrm{#1}}
\begin{description}
\setlength{\labelwidth}{13mm}
\setlength{\labelsep}{1.5mm}
\setlength{\itemindent}{0mm}

\item[$(1)$] If $[Q_{A,B} (s,t,h)] \in W (C,D)$ for some $(C,D) \in \Sigma (X,L)$, then $(C,D)=(A,B)$.

\item[$(2)$] If $[Q_{A,B} (s,t,h)] = [Q_{A,B} (u,v,w)]$ in $\P \left( I(X,L)_2 \right)$
for some $(u,v,w) \in V(A,B)$, then $\langle s, t \rangle = \langle u,v \rangle$ and $\langle h \rangle = \langle w \rangle$.
\end{description}
\end{proposition}

\begin{proof}
Suppose that $[Q_{A,B} (s,t,h)] \in W(C,D)$ for some $(C,D) \in \Sigma (X,L)$. That is,
\begin{equation*}
Q_{A,B} (s,t,h)= Q_{C,D} (u,v,w)
\end{equation*}
for some $u,v \in H^0 (X,C)$ and $w \in H^0 (X,D)$. Then we get the equality
\begin{equation*}
G := \varphi (s^2 h) \varphi (t^2 h) - \varphi (u ^2 w) \varphi (v^2 w)   = \left( \varphi (st h)-\varphi (uv w) \right) \left( \varphi (s t h)+\varphi (uv w) \right)
\end{equation*}
We will show that $(A,B) = (C,D)$. Consider the subspace
$$V = \langle \varphi (s^2 h) , \varphi(t^2 h) , \varphi(u^2 w ) ,\varphi (v^2 w) \rangle$$
of $H^0 (\P^r , \mathcal{O}_{\P^r} (1))$. Note that
$$2 \leq \dim_{\mathbb{K}}~ V \leq 3$$
since $\{ \varphi (s^2 h) , \varphi(t^2 h) \}$ and $\{ \varphi(u^2 w ) ,\varphi (v^2 w)\}$ are respectively $\mathbb{K}$-linearly independent and the rank of $G$ is at most $2$.

If $\dim_{\mathbb{K}} V = 2$ then $V = \langle \varphi (s^2 h) , \varphi(t^2 h) \rangle = \langle \varphi(u^2 w ) ,\varphi (v^2 w)\rangle$ and hence
$$\langle  s^2 h , t^2 h \rangle = \langle u^2 w   , v^2 w \rangle$$
since $\varphi : H^0 (X,L) \rightarrow H^0 (\P^r , \mathcal{O}_{\P^r} (1))$ is an isomorphism. Then we get
$$h = \gcd (s^2 h , t^2 h) = \gcd (u^2 w , v^2 w ) = w \times \gcd (u^2 ,v^2 ).$$
If $B = \mathcal{O}_X$ then $h$ is a nonzero constant and hence so is $w$. Thus
\begin{equation*}
D = \mathcal{O}_X ~\mbox{and}~ L = A^2 = C^2 .
\end{equation*}
In consequence, $(A,B) = (C,D)$ since ${\rm Pic}(X)$ is torsion free. Next, if $h$ is irreducible then $h = cw$ for some $c \in \mathbb{K}^*$ and hence $B=D$. Again this implies that $(A,B) = (C,D)$.

Now, suppose that $\dim_{\mathbb{K}} V = 3$. We may assume that
$$V = \langle \varphi (s^2 h) , \varphi(t^2 h) , \varphi(u^2 w ) \rangle$$
and
$$\varphi(v^2 w ) = a\varphi (s^2 h) +b \varphi(t^2 h) +c \varphi(u^2 w )$$
for some $a,b,c \in \mathbb{K}$. Since $V$ is $\mathbb{K}$-linearly independent and $G$ is of rank at most $2$, it must be true that $c+ab=0$ and $G$ is factored into
$$G = \left( a\varphi (u^2 w)-\varphi (t^2 h) \right) \left( b \varphi (u^2 w)-\varphi (s^2 h) \right).$$
Therefore there is an element $\lambda \in \mathbb{K}^*$ such that either
$$(i)~ \begin{cases}
a\varphi (u^2 w)-\varphi (t^2 h) =   \lambda \varphi (s  t h)- \lambda \varphi (u v w) \\
b \varphi (u^2 w)-\varphi (s^2 h) =  \lambda^{-1} \varphi (st h)+ \lambda^{-1} \varphi (uv w)
\end{cases}$$
or else
$$(ii)~ \begin{cases}
a\varphi (u^2 w)-\varphi (t^2 h) =   \lambda \varphi (sth)+ \lambda \varphi (uvw) \\
b \varphi (u^2 w)-\varphi (s^2 h) =  \lambda^{-1} \varphi (sth)- \lambda^{-1} \varphi (uvw)
\end{cases}.$$
In case $(i)$, it follows that
$$(i)' ~ \begin{cases}
a u^2 w - t^2 h  =   \lambda sth  - \lambda uvw \\
b u^2 w - s^2 h  =  \lambda^{-1} sth + \lambda^{-1}  uvw
\end{cases}$$
and hence, by eliminating the term $uvw$, we get
\begin{equation}\label{eqn:useful factorization 1}
(\lambda s + t)^2 h = (a+\lambda^2 b )u^2 w.
\end{equation}
Also from $(i)'$ we get the equalities
\begin{equation*}
th (\lambda s + t) = uw (au+\lambda v) ~ \mbox{and} ~ sh (\lambda s +t) = uw (\lambda bu + v ).
\end{equation*}
Thus we get
\begin{equation*}
sth (\lambda s + t) = suw (au+\lambda v) =tuw (\lambda bu + v )
\end{equation*}
and hence
\begin{equation}\label{eqn:useful factorization 2}
s (au+\lambda v) =t (\lambda bu + v ).
\end{equation}
This implies that $s$ and $t$ divide respectively $\lambda bu + v$ and $au+\lambda v$ since they are irreducible and $\mathbb{K}$-linearly independent. In particular, it holds that
\begin{equation*}
C = A \otimes E ~\mbox{for some}~E \in {\rm Pic}(X)~\mbox{with}~h^0 (X,E) \geq 1.
\end{equation*}
Then $B = E^2 \otimes D$. If $B = \mathcal{O}_X$ then $E = D = \mathcal{O}_X$ since $E$ and $D$ have nonzero sections. If $h$ is irreducible then it follows from (\ref{eqn:useful factorization 1}) that $h$ divides $w$ and hence
\begin{equation*}
D = B \otimes F ~\mbox{for some}~F \in {\rm Pic}(X)~\mbox{with}~h^0 (X,F) \geq 1.
\end{equation*}
Then $E^2 \otimes F = \mathcal{O}_X$ and so it follows that $E = F = \mathcal{O}_X$ since $E$ and $F$ have nonzero sections. In consequence we obtain the desired equality $(A,B) = (C,D)$. We can deal case $(ii)$ in a similar way.

\noindent $(2)$ We follow the proof of $(1)$.

Suppose that $\dim_{\mathbb{K}} V = 2$. If $B = \mathcal{O}_X$ then $h$ and $w$ are nonzero constants. If $h$ is irreducible then $h = cw$ for some $c \in \mathbb{K}^*$. Therefore it holds that $\langle s^2 , t^2 \rangle = \langle u^2 ,v^2 \rangle$. Then one can show the desired equality $\langle s , t \rangle = \langle u ,v \rangle$.

Suppose that $\dim_{\mathbb{K}} V = 3$. For case $(i)$ in the proof of $(1)$, it follows from (\ref{eqn:useful factorization 2}) that $s$ and $t$ are constant multiples of $au+\lambda v$ and $\lambda bu + v $ since they are irreducible and $\mathbb{K}$-linearly independent. Thus we get the desired equality $\langle s , t \rangle = \langle u ,v \rangle$. The case $(ii)$ in the proof of $(1)$ can be dealt in a similar way.
\end{proof}

\begin{definition}
Let $(X,L)$ be as above. We say that an element $(A,B) \in \Sigma (X,L)$ is \textit{irreducible} if a general member of $H^0 (X,A)$ is irreducible in $\mbox{Cox}(X)$ and if either $B = \mathcal{O}_X$ or else a general member of $H^0 (X,B)$ is irreducible in $\mbox{Cox}(X)$.
\end{definition}

\begin{corollary}\label{cor:generically one to one}
Let $(A,B)$ be an irreducible member of $\Sigma (X,L)$ such that $p = h^0 (X,A)-1$ and $q = h^0 (X,B)-1$. Then the morphism
$$\widetilde{Q_{A,B}} : \mathbb{G} (1,\P^p ) \times \P^q \rightarrow \P ( I(X,L)_2 )$$
is generically injective. In particular, the degree of the subvariety $W(A,B) \subset \P ( I(X,L)_2 )$ is given by
\begin{equation*}
\deg \left( W(A,B)\right) = \deg \left( \mathcal{O}_{\mathbb{G}(1,\P^p ) \times \P^q } (2,2) \right)  = \frac{ 2^{2p+q-2}}{ p}  {{2p+q-2} \choose {2p-2}}  {{2p-2} \choose {p-1}}.
\end{equation*}
\end{corollary}

\begin{corollary}\label{cor:generically one to one}
Let $(A,B)$ be an irreducible member of $\Sigma (X,L)$ such that $p = h^0 (X,A)-1$ and $q = h^0 (X,B)-1$. Then the morphism
$$\widetilde{Q_{A,B}} : \mathbb{G} (1,\P^p ) \times \P^q \rightarrow \P ( I(X,L)_2 )$$
is generically injective. In particular, the degree of the subvariety $W(A,B) \subset \P ( I(X,L)_2 )$ is given by
\begin{equation*}
\deg \left( W(A,B)\right) = \deg \left( \mathcal{O}_{\mathbb{G}(1,\P^p ) \times \P^q } (2,2) \right)  = \frac{ 2^{2p+q-2}}{ p}  {{2p+q-2} \choose {2p-2}}  {{2p-2} \choose {p-1}}.
\end{equation*}
\end{corollary}

\begin{proof}
Let $\mathcal{U} \subset V(A,B)$ be an open subset such that every member of $\mathcal{U}$ satisfies the above two conditions. Then Proposition \ref{prop:distinguishing}.$(2)$ shows that $\widetilde{Q_{A,B}}$ is injective on $\mathcal{U}$ and hence it is generically injective. By combining this result with Theorem \ref{thm:main1 rank3 case}, we have
\begin{equation*}
\deg \left( W(A,B) \right) =  \deg \left( \mathcal{O}_{\mathbb{G} (1,\P^p ) \times \P^q }(2,2) \right) .
\end{equation*}
Since ${\rm dim} ~\mathbb{G} (1,\P^p ) \times \P^q = 2p+q-2$ and $\mathcal{O}_{\mathbb{G} (1,\P^p ) \times \P^q }(2,2) = \mathcal{O}_{\mathbb{G} (1,\P^p ) \times \P^q }(1,1)^2$, we have
\begin{equation*}
\deg \left( \mathcal{O}_{\mathbb{G} (1,\P^p ) \times \P^q }(2,2) \right) = 2^{2p+q-2} \times \deg \left( \mathcal{O}_{\mathbb{G} (1,\P^p ) \times \P^q }(1,1) \right).
\end{equation*}
Also, it holds that
\begin{equation*}
\deg \left( \mathcal{O}_{\mathbb{G} (1,\P^p ) \times \P^q }(1,1) \right) =  {{2p+q-2} \choose {2p-2}} \times \deg \left( \mathcal{O}_{\mathbb{G} (1,\P^p )}(1) \right) \times \deg \left( \mathcal{O}_{ \P^q }(1) \right).
\end{equation*}
since the Hilbert polynomial of $(\mathbb{G} (1,\P^p ) \times \P^q , \mathcal{O}_{\mathbb{G} (1,\P^p ) \times \P^q }(1,1) )$ is equal to the product of those of $( \mathbb{G} (1,\P^p ) , \mathcal{O}_{\mathbb{G} (1,\P^p )}(1) ))$ and $( \P^q , \mathcal{O}_{ \P^q }(1) )$. Finally, it holds that
\begin{equation*}
 \deg \left( \mathcal{O}_{\mathbb{G} (1,\P^p )}(1) \right) = \frac{1}{p} \times {{2p-2} \choose {p-1}}
\end{equation*}
(cf. \cite[Chapter 19, Page 247]{Harris}). This proves the formula of $\deg (W(A ,B) )$.
\end{proof}

\noindent {\bf Proof of Theorem \ref{thm:main2}.} By Theorem \ref{thm:main1}.$(2)$, the following is an irreducible decomposition of $\Phi_3 (X,L)$.
\begin{equation*}
\Phi_3 (X,L) ~=~ \bigcup_{(A,B) \in \Sigma (X,L)} W(A,B)
\end{equation*}
Now suppose that $(A,B) \in \Sigma (X,L)$ is irreducible. Then Proposition \ref{prop:distinguishing}.$(1)$ implies that if $(C,D) \neq (A,B)$ then $W(C,D)$ does not contain $W(A,B)$. This shows that $W(A,B)$ is an irreducible component of $\Phi_3 (X,L)$. In particular, if every $(A,B) \in \Sigma (X,L)$ is irreducible then the above irreducible decomposition of $\Phi_3 (X,L)$ is already minimal.      \qed \\

The rest of this section will be devoted to describing $\Phi_3 (V_{n,2})$ where
$$V_{n,2} = \nu_2 (\P^n ) \subset \P^{r},~r={{n+2} \choose {2}}-1,$$
is the second Veronese variety corresponding to $(\P^n , \mathcal{O}_{\P^n} (2))$.

\begin{notation and remarks}
$(1)$ Let $A$ be the very ample line bundle on the Grassmannian manifold $\mathbb{G} (1,\P^n )$ which generates the group $\mbox{Pic} (\mathbb{G} (1,\P^n ))$. Also for each $d \geq 1$, let
\begin{equation*}
\mathbb{G}_d \subset \P H^0 (\mathbb{G} (1,\P^n ),A^d )
\end{equation*}
be the linearly normal embedding of $\mathbb{G} (1,\P^n )$ by $A^d$. For example,
$$\mathbb{G}_1 \subset \P H^0 (\mathbb{G} (1,\P^n ),A) = \P^{{{n+1} \choose {2}}-1}$$
is called the Pl\"{u}cker embedding of $\mathbb{G} (1,\P^n )$. Note that
\begin{equation}\label{eqn:dimensions}
h^0 (\mathbb{G} (1,\P^n ),A^2 ) = {{n+1} \choose {2}} {{n+2} \choose {2}} = \dim_{\mathbb{K}} ~I(V_{n,2} )_2  .
\end{equation}
Thus $\mathbb{G}_2$ and $\Phi_3 (V_{n,2})$ are contained in the same projective space.
\end{notation and remarks}
\smallskip

\begin{theorem}\label{thm:2nd Veronese}
Suppose that $\mbox{char}(\mathbb{K}) \neq 2,3$. Then $\Phi_3 (V_{n,2} ) \subset \P \left( I(V_{n,2} )_2 \right)$ is projectively equivalent to $\mathbb{G}_2 \subset \P H^0 (\mathbb{G} (1,\P^n ),A^2 )$.
\end{theorem}

\begin{proof}
Theorem \ref{thm:main2} shows that $\Phi_3 (V_{n,2} )$ is irreducible since its minimal irreducible decomposition is
$$\Phi_3 (V_{n,2} ) = W(\mathcal{O}_{\P^n} (1), \mathcal{O}_{\P^n} ).$$
Also Theorem \ref{thm:main1 rank3 case} shows that there is a finite morphism
\begin{equation*}
Q_{\mathcal{O}_{\P^n} (1) ,\mathcal{O}_{\P^n} } : \mathbb{G} (1,\P^n ) \rightarrow \P \left( I(V_{n,2} )_2 \right) =  \P^{{{n+1} \choose {2}} {{n+2} \choose {2}}-1}.
\end{equation*}
which is defined by a linear subsystem of $|A^2 |$. The image of this morphism is
$$W(\mathcal{O}_{\P^n} (1), \mathcal{O}_{\P^n} )$$
which is nondegenerate in $ \P \left( I(V_{n,2} )_2 \right)$ by \cite[Theorem 1.1]{HLMP}.
By this fact and (\ref{eqn:dimensions}), we can conclude that the above morphism is defined by the complete linear system $|A^2 |$. Therefore $W(\mathcal{O}_{\P^n} (1), \mathcal{O}_{\P^n} )$ is equal to $\mathbb{G}_2$.
\end{proof}

\section{Examples}
\noindent Let $X \subset \mathbb{P}^r$ be a linearly normal projective variety defined by a very ample line bundle $L$ on $X$. By Theorem \ref{thm:main1} and Theorem \ref{thm:main2}, we obtain the irreducible decomposition of $\Phi_3 (X,L)$ by the projective varieties $W(A,B)$'s for $(A,B) \in \Sigma (X,L)$ when $X$ is locally factorial and ${\rm Pic} (X)$ is finitely generated. Along this line, this section is devoted to providing examples about the case where $X$ is an irrational curve and hence ${\rm Pic}(X)$ is not finitely generated.\\

Let $C \subset \P^4$ be an elliptic normal curve of degree $5$. Then $I(C)$ is generated by $5$ quadratic polynomials. Moreover, it satisfies property $\QR (3)$ by Theorem 1.1 in \cite{PaCurve}. If $\mathcal{O}_C (1) = A^2 \otimes B$ for some $A,B \in \mbox{Pic}(C)$ such that $h^0 (C,A) \geq 2$ and $h^0 (C,B) \geq 1$, then
$$B = \mathcal{O}_C (x) ~ \mbox{for some} ~ x \in C ~ \mbox{and $A$ is a line bundle of degree $2$.}$$
Conversely, for any $x \in C$ the line bundle $\mathcal{O}_C (1) \otimes \mathcal{O}_C (-x)$ is a square of a line bundle of degree $2$. Thus it holds that
\begin{equation*}
\Sigma (C,\mathcal{O}_C (1)) = \{ (A,\mathcal{O}_C (x)) ~|~ A^2 = \mathcal{O}_C (1) \otimes \mathcal{O}_C (-x) \}.
\end{equation*}
In particular, there is a $4:1$ map from $\Sigma (C,\mathcal{O}_C (1))$ to $C$. Theorem \ref{thm:main1} says that for any $(A,B) \in \Sigma (C,\mathcal{O}_C (1))$, the set $W(A,B)$ is the image of
$$\mathbb{G} (1,\P H^0 (C,A)) \times \P H^0 (C,B) = \mathbb{G} (1,\P^1 ) \times \P^0$$
and hence it is a single point. In the following example, we will see that $\Phi_3 (C)$ is an irreducible curve. In consequence, the irreducible components of $\Phi_3 (C)$ are not of the form $W(A,B)$ for some $(A,B) \in \Sigma (C,\mathcal{O}_C (1))$.

\begin{example}
For the smooth rational normal surface scroll
$$S = S(1,2) = \{[sx:tx:s^2 y :sty:t^2 y] ~|~s,t,x,y \in \mathbb{K} \} \subset \P^4 ,$$
the homogeneous ideal $I(S)$ is generated by the $2 \times 2$ minors of the matrix
$$\left(\begin{array}{rrrr}
z_0 & z_2 &  z_3   \\
z_1 & z_3 &  z_4
\end{array}\right).$$
Namely, $I(S) = \langle Q_1 , Q_2 , Q_3 \rangle$ where
\begin{equation*}
\begin{CD}
Q_1 = z_0 z_3 -z_1 z_2 ,~Q _{2} =z _{0} z _{4} -z _{1} z_3 \quad \mbox{and} \quad Q _{3} =z _{2} z _{4} -z _3 ^2 .
\end{CD}
\end{equation*}
Let $C$ be a divisor of $S$ defined by $sx^2 + (s^2 +t^2 ) xy + (s^3 +  t^3 )y^2$ in $S$. Then one can check the followings:\\

\renewcommand{\descriptionlabel}[1]%
             {\hspace{\labelsep}\textrm{#1}}
\begin{description}
\setlength{\labelwidth}{13mm}
\setlength{\labelsep}{1.5mm}
\setlength{\itemindent}{0mm}

\item[$(i)$] $C \equiv 2H-F$ where $H$ and $F$ denote respectively the hyperplane section and a ruling of $S$.

\item[$(ii)$] $I(C) = \langle Q_1 , Q_2 , Q_3 ,Q_4 , Q_5 \rangle$ where
\begin{equation*}
\quad \quad Q_4 = z_0 ^2  + z_0 z_2 +  z_0 z_4 + z_2 ^2  +  z_3 z_4 \quad \mbox{and} \quad Q _{5} = z_0 z_1   + z_1 z_2 +  z_1 z_4 + z_2 z_3   +   z_4 ^2 .
\end{equation*}

\item[$(iii)$] $C$ is smooth and hence $C \subset \P^4$ is an elliptic normal curve of degree $5$.\\
\end{description}
The $5 \times 5$ symmetric matrix $M$ associated to $Q  = x_1 Q_1 + x_2 Q_2 + x_3 Q_3 + x_4 Q_4 + x_5 Q_5$
is written as
$$M  = \left(\begin{array}{rrrrrr}
2x_4   & x_5   & x_4 & x_1    & x_2 +x_4    \\
x_5 & 0 & x_5 -x_1   & -x_2 & x_5   \\
x_4 & x_5 -x_1   & 2x_4   & x_5 & x_3    \\
x_1   & -x_2 & x_5 & -2x_3 & x_4   \\
x_2 +x_4 & x_5 & x_3 & x_4 & 2x_5
\end{array}\right).$$
Then it can be checked by ``Macaulay 2" that $\sqrt{I(4,M)}$ is a homogeneous prime ideal which is generated by $5$ cubic polynomials and $10$ quartic polynomials. Furthermore it holds that
$$\Phi_3 (C) = V(\sqrt{I(4,M)}) \subset \P^4$$
is a smooth elliptic curve of degree $10$. In particular, $\Phi_3 (C)$ is an irreducible curve.
\end{example}

\end{document}